\documentclass[a4paper]{amsart}

\usepackage[dvipsnames,svgnames,x11names,hyperref]{xcolor}
\usepackage{textgreek,bbm,url,graphicx,verbatim,amssymb,enumerate,stmaryrd,booktabs,lmodern,mathtools,microtype,mathabx}
\usepackage[pagebackref,colorlinks,citecolor=Mahogany,linkcolor=Mahogany,urlcolor=Mahogany,filecolor=Mahogany]{hyperref}
\usepackage[capitalize]{cleveref}
\usepackage[mathscr]{euscript}
\usepackage[margin=1.5in]{geometry}
\usepackage{tikz,tikz-cd}
\usetikzlibrary{matrix,calc,positioning,arrows,decorations.pathreplacing,patterns,arrows}

\newtheorem{theorem}{Theorem}[section]
\newtheorem{lemma}[theorem]{Lemma}
\newtheorem{proposition}[theorem]{Proposition}
\newtheorem{corollary}[theorem]{Corollary}
\newtheorem*{corollary*}{Corollary}

\newtheorem{atheorem}{Theorem}

\newtheorem{acorollary}[atheorem]{Corollary}

\theoremstyle{definition}
\newtheorem{definition}[theorem]{Definition}

\newtheorem{convention}[theorem]{Convention}

\theoremstyle{remark}

\newtheorem{remark}[theorem]{Remark}
\newtheorem{warning}[theorem]{Warning}
\newtheorem{question}[theorem]{Question}

\newcommand{\cat}[1]{{\mathsf{#1}}}
\newcommand{\icat}[1]{{\mathscr{#1}}}
\renewcommand{\rm}[1]{{\mathrm{#1}}}
\newcommand{\lra}{\longrightarrow}

\newcommand{\bb}[1]{{\mathbb{#1}}}
\newcommand{\cal}[1]{{\mathcal{#1}}}

\newcommand{\Map}{\rm{Map}}

\newcommand{\hAut}{\rm{hAut}}
\newcommand{\Homeo}{\rm{Homeo}}
\newcommand{\Diff}{\rm{Diff}}
\newcommand{\Aut}{\rm{Aut}}
\newcommand{\Hom}{\rm{Hom}}
\newcommand{\fin}{\rm{fin}}
\newcommand{\dgLie}{{\rm{dg}\icat{L}\rm{ie}}}
\newcommand{\dgCom}{{\rm{dg}\icat{C}\rm{om}}}
\newcommand{\aug}{\rm{aug}}

\newcommand{\Ch}{{\icat{C}\rm{h}}}

\DeclareFontFamily{U}{min}{}
\DeclareFontShape{U}{min}{m}{n}{<-> udmj30}{}

\AtBeginDocument{%
	\def\MR#1{}
}

\title{Mapping class groups of manifolds with boundary are of finite type}

\author{Alexander Kupers}
\address{Department of Computer and Mathematical Sciences, University of Toronto Scarborough, 1265 Military Trail, Toronto, ON M1C 1A4, Canada}
\email{a.kupers@utoronto.ca}

\date{\today}

\begin{document}
	
\begin{abstract}In this note we prove that the mapping class group of a compact topological manifold $M$ with boundary is of finite type, under assumptions on its dimension and connectivity.
\end{abstract}
	
\maketitle

\section{Introduction}

\begin{definition}\qquad
	\begin{itemize}
		\item We say that a space $X$ is of \emph{finite type} if it is homotopy equivalent to a CW-complex with finitely many cells in each dimension.
		\item We say that a group $G$ is of \emph{finite type} (equivalently, is of \emph{type $F_\infty$}) if its classifying space $BG$ is of finite type.
	\end{itemize}
\end{definition}

For compact topological manifold $M$ with boundary $\partial M$, let $\Homeo_\partial(M)$ denote the topological group of homeomorphisms of $M$ fixing a neighbourhood of the boundary pointwise, in the compact-open topology. The \emph{mapping class group} of $M$ is its group $\pi_0\,\Homeo_\partial(M)$ of path components.

\begin{atheorem}\label{athm:main} If $M$ is a compact topological manifold of dimension $\geq 5$, and both $M$ and all path components of $\partial M$ are $1$-connected, then the group $\pi_0\,\rm{Homeo}_\partial(M)$ is of finite type.
\end{atheorem}

\begin{remark}\label{rem:mainthm} \quad
	\begin{enumerate}[(i)]
	\item This paper crucially uses \cite{EspicSaleh}, which studies relative homotopy automorphisms. Apart from this, the arguments are standard with the exception of \cref{cor:borel-lemma}.
	\item \label{enum:sullivan} For \emph{closed} $1$-connected topological manifolds of dimension $\geq 5$, \cref{athm:main} is a consequence of a stronger result of Sullivan \cite[Theorem 13.3]{Sullivan}, who proved $\pi_0\, \Homeo(M)$ is commensurable up to finite kernel with an arithmetic group (he only did the case $d\geq 6$ but nowadays we know $d\geq 5$ suffices, see the addendum to \cite[Theorem 10.3]{Wall}). Here \emph{commensurability up to finite kernel} is the equivalence relation on groups generated by isomorphism, passing to finite index subgroups, and taking quotients by finite normal subgroups.
	\item \label{enum:stronger} If $\partial M$ is connected, we prove the stronger statement that $\pi_0\,\Homeo_\partial(M)$ is commensurable up to finite kernel to an extension of an arithmetic group by a finitely generated nilpotent group; see \cref{thm:ctd-case}.
	\item \label{enum:connectivity} We can drop the assumption that the path components of $\partial M$ are 1-connected at the expense of additional conditions on $M$; see \cref{cor:nontrivial-pi1}.
	\item \label{enum:smooth} If $M$ admits a smooth or PL structure, the analogue of \cref{athm:main} is true for diffeomorphisms or PL homeomorphisms fixing a neighbourhood of the boundary pointwise.
	\item \label{enum:dim4} \cref{athm:main} is also true in dimension $d \leq 4$; the only non-classical case is $d=4$ in which the result follows from \cite[Theorem 1.1]{Quinn}. However, we do \emph{not} know whether it holds for diffeomorphisms or PL homeomorphisms when $d=4$.
	\end{enumerate}
\end{remark}

\cref{rem:mainthm} \eqref{enum:connectivity}, \eqref{enum:smooth}, and \eqref{enum:dim4} will be explained in \cref{sec:remarks}. Combining \cref{athm:main} with work of Bustamante, Krannich, and author (in particular \cite[Remark 5.2]{BustamanteKrannichKupers}) we obtain the analogue of \cite[Theorem A]{BustamanteKrannichKupers} for certain manifolds with boundary:

\begin{acorollary}\label{acor:finite-type} If $M$ is a compact smooth manifold of dimension $2n \geq 6$, and both $M$ and all path components of $\partial M$ are $1$-connected, then $B\Homeo_\partial(M)$ and all of its homotopy groups are of finite type.\end{acorollary}

\begin{remark}\qquad
	\begin{enumerate}[(i)]
		\item As before, the same is true for diffeomorphisms and PL homeomorphisms fixing a neighbourhood of the boundary pointwise.
		\item The condition that $M$ is smooth is due to the fact that \cite{BustamanteKrannichKupers} relies on convergence of embedding calculus, which is only known for smoothable manifolds.
	\end{enumerate}
\end{remark}

\subsection{Open questions}
Two questions remain. Firstly, Sullivan's result mentioned in \cref{rem:mainthm} \eqref{enum:sullivan} suggests the following:

\begin{question}\label{ques:arithmetic}Under the assumptions of \cref{athm:main}, is $\pi_0\,\rm{Homeo}_\partial(M)$ commensurable up to finite kernel with an arithmetic group?\end{question}

Secondly, Triantafillou extended Sullivan's result to closed smooth manifolds of even dimensions $\geq 6$ with \emph{finite} fundamental group \cite{TriantafillouFinite}; her proof contains some flaws, which are fixed in \cite[Section 2.2]{BustamanteKrannichKupers}. Combined with \cref{rem:mainthm} \eqref{enum:connectivity}, this suggests the following:

\begin{question}Do \cref{athm:main} and \cref{acor:finite-type} remain true when we assume that $M$ is $0$-connected with finite fundamental group and make no assumptions on $\partial M$?\end{question}
		
\subsubsection*{Acknowledgements} AK would like to thank Manuel Krannich and Mark Powell for helpful conversations, Hadrien Espic and Bashar Saleh for answering my questions, and the anonymous referee for many comments and suggestions. AK acknowledges the support of the Natural Sciences and Engineering Research Council of Canada (NSERC) [funding reference number 512156 and 512250], as well as the Research Competitiveness Fund of the University of Toronto at Scarborough and the Alfred J.~Sloan Research Fellowship.

\tableofcontents

\section{Linear algebraic groups and arithmetic groups} This section contains results regarding linear algebraic and arithmetic groups necessary for the proof of \cref{athm:main}. As the reader may not be familiar with these results, we give more background than strictly necessary (the reader may find the summary on \cite[p.\,294--296]{Sullivan} helpful as well).

\subsection{Linear algebraic groups}
Our definitions are as in \cite[Chapter I]{Borel}. An \emph{algebraic $\bb{Q}$-group} is a group object in the category whose objects are varieties over $\bb{Q}$ and whose morphisms are morphisms of varieties defined over $\bb{Q}$. A \emph{morphism of linear algebraic groups over $\bb{Q}$} is a morphism of such group objects. The prototypical example is $\rm{GL}_V$ for $V$ a finite-dimensional $\bb{Q}$-vector space; we write $\rm{GL}_n$ when $V = \bb{Q}^n$. A \emph{closed subgroup} of an algebraic $\bb{Q}$-group is a (Zariski) closed subvariety that is an algebraic $\bb{Q}$-group with the restricted operations.

\begin{definition}A \emph{linear algebraic group over $\bb{Q}$} is an algebraic $\bb{Q}$-group $G$ over $\bb{Q}$ isomorphic to a closed subgroup of $\rm{GL}_V$.\end{definition}

Concretely, this means $G$ can be described as the vanishing locus in $\rm{GL}_n$ of finitely many polynomials with coefficients in $\bb{Q}$ in the entries of the matrix; this is how we will often construct linear algebraic groups. In fact, every affine algebraic $\bb{Q}$-group is of this form \cite[Proposition I.1.10]{Borel}. 

\begin{convention}\label{conv:q}
	All our linear algebraic groups and morphisms of linear algebraic groups will be defined over $\bb{Q}$, so we will forego mentioning this.
\end{convention}

We will need to know that the following constructions yield linear algebraic groups, clear from the definitions unless mentioned otherwise.

\begin{lemma}\label{lem:alg-grp-constr} The following constructions yield linear algebraic groups:
	\begin{enumerate}[\noindent (1)]
		\item \label{enum:intersection-alg-subgrp} Intersections of finitely many closed subgroups.
		\item \label{enum:kernels-alg} Kernels of morphisms of linear algebraic groups.
		\item \label{enum:images-alg} Images of morphisms of linear algebraic groups, by \cite[Corollary I.6.9]{Borel}.
		\item \label{enum:quotients-alg} Quotients by normal closed subgroups, by \cite[Theorem I.6.8]{Borel}.
		\item \label{enum:stabilisers-alg} Stabilisers of an element in an algebraic representation as in \cref{def:alg-action}.
	\end{enumerate}
\end{lemma}

\begin{definition}\label{def:alg-action} An \emph{algebraic representation} of a linear algebraic group $G$ on a finite-dimensional $\bb{Q}$-vector space $V$ is a morphism $G \to \rm{GL}_V$ of linear algebraic groups.\end{definition}

We will also refer to this as an \emph{algebraic action}. Equivalently, the adjoint maps $G \times V \to V$ satisfy the usual equations for an action and are morphisms of varieties. We will need that the following constructions yield algebraic representations, all clear from the definitions:

\begin{lemma}\label{lem:alg-act-constructions} The following constructions yield algebraic representations:
	\begin{enumerate}[\noindent (1)]
		\item \label{enum:alg-rep-sub-quot} Subrepresentations and quotients of algebraic representations, in the following sense: if $V$ is an algebraic representation and $W \subseteq V$ is preserved by the action, then $W$ and $V/W$ are also algebraic representations.
		\item \label{enum:alg-rep-tensor} Tensor products, exterior powers, and symmetric powers of algebraic representations.
		\item \label{enum:alg-rep-dual} Linear duals of algebraic representations.
	\end{enumerate}
\end{lemma}

Similarly, we can define an action of one algebraic group $G$ on another algebraic group $H$ by group automorphisms. In that case there is another construction of linear algebraic groups \cite[Section I.1.11]{Borel}:

\begin{lemma}\label{lem:semi-direct} Given an action of a linear algebraic group $G$ on a linear algebraic group $H$, then there is a linear algebraic group $G \ltimes H$ so that there is a split exact sequence
	\[1 \lra H \lra G \ltimes H \lra G \lra 1\]
with induced action of $G$ on $H$ as given.
\end{lemma}

If $G$ is a linear algebraic group, its $\bb{Q}$-points $G(\bb{Q})$ form a group. If $1 \to G \to G' \to G' \to 1$ is a short exact sequence of linear algebraic groups, then the sequence
	\[1 \lra G(\bb{Q}) \lra G'(\bb{Q}) \lra G''(\bb{Q})\]
is exact; we learned this from \cite[Section 9]{Wilkerson} but for proofs see \cite{SerreGalois}. In this note, we will need two situations in which this sequence extends to the right to a short exact sequence. The first is obvious; if the sequence of linear algebraic groups is split exact so is 
\[1 \lra G(\bb{Q}) \lra G'(\bb{Q}) \lra G''(\bb{Q}) \lra 1.\]
The second will be \cref{lem:unipotent-q-points}.

\subsubsection{Arithmetic groups}
If we choose an identification of a linear algebraic group $G$ as a closed subgroup of $\rm{GL}_n$ then we may define
\[G_\bb{Z} \coloneqq G(\bb{Q}) \cap \rm{GL}_n(\bb{Z}).\]
This depends on our choice, but by \cite[Section 6]{BorelHarishChandra} any two choices yield subgroups that are equivalent in the following sense: two subgroups $H,H' \subseteq G$ are \emph{commensurable} if $H \cap H'$ has finite index in both $H$ and $H'$. Thus it makes sense to say that subgroup $A \subset G(\bb{Q})$ is an \emph{arithmetic subgroup} if it is commensurable to $G_\bb{Z}$.

\begin{definition}A group is \emph{arithmetic} if it is isomorphic to an arithmetic subgroup of the $\bb{Q}$-points of a linear algebraic group.\end{definition}

We will need that arithmetic groups are finitely generated \cite[1.3(1)]{Serre}, in fact even of finite type \cite[2.3]{Serre}, and that arithmetic subgroups are preserved by the following constructions:

\begin{lemma}\label{lem:arithm-constructions}The following constructions yield arithmetic subgroups:
	\begin{enumerate}[\noindent (1)]
		\item \label{enum:image-arithm} Images under surjective morphisms of linear algebraic groups, by \cite[1.1]{Serre}: if $\phi \colon G \to G'$ is a surjective morphism of linear algebraic groups and $A \subset G(\bb{Q})$ is an arithmetic subgroup, then so is $\phi(A) \subset G'(\bb{Q})$.
		\item \label{enum:intersection-arithm} Intersections with closed subgroups: if $H \subseteq G$ is a closed subgroup and $A \subset G(\bb{Q})$ is an arithmetic subgroup, then so is $A \cap H(\bb{Q}) \subset H(\bb{Q})$.
	\end{enumerate}
\end{lemma}

\subsection{Unipotent groups} \label{sec:unipotent} We shall take \cite[I.4.8 Corollary]{Borel} as the definition of a unipotent linear algebraic group. Let $\rm{U}_n \subset \rm{GL}_n$ be the closed subgroup of matrices with $0$'s below the diagonal and $1$'s on the diagonal.

\begin{definition}
	A linear algebraic group is \emph{unipotent} if it is isomorphic to a closed subgroup of $\rm{U}_n$.
\end{definition}

The $\bb{Q}$-points of unipotent linear algebraic groups have two useful properties. Firstly, they admit a unique such structure up to isomorphism of linear algebraic groups \cite[p.\,295]{Sullivan}, a reference is \cite[Corollary 3.8 (1)]{Chatterjee}. Secondly, the proof of \cite[Theorem 9.5]{Wilkerson} (not affected by \cite{WilkersonErrata}) yields:

\begin{lemma}\label{lem:unipotent-q-points}If $1 \to G \to G' \to G'' \to 1$ is an exact sequence of linear algebraic groups and $G$ is unipotent, then the following sequence of groups is also exact:
	\[1 \lra G(\bb{Q}) \lra G'(\bb{Q}) \lra G''(\bb{Q}) \lra 1.\]
\end{lemma}

By work of Malcev, if $U$ is a unipotent linear algebraic group then the group $U(\bb{Q})$ is essentially determined by its nilpotent Lie algebra $\mathfrak{u}$ over $\bb{Q}$ \cite[Chapter 6]{Segal}, through mutually inverse functions we will describe shortly:
\begin{equation}\label{eqn:log-exp} \begin{tikzcd} \mathfrak{u} \rar[swap]{\exp} & U(\bb{Q}) \arrow[shift left=-1ex,swap]{l}{\log}.\end{tikzcd}\end{equation}
Choosing an inclusion $U \hookrightarrow \rm{U}_n$ we see that $U(\bb{Q})$ lies in the subgroup $\rm{U}_n(\bb{Q}) \subset \rm{GL}_n(\bb{Q})$, though in this section we follow Segal and write $\rm{Tr}_n(1,\bb{Q})$ for $\rm{U}_n(\bb{Q})$ to make it easier for the reader to consult this reference. As $\rm{Tr}_n(1,\bb{Q})$ is the group of upper-triangular matrices with entries in $\bb{Q}$ and $1$'s on the diagonal, its Lie algebra is given by the Lie algebra $\rm{Tr}_n(0,\bb{Q})$ of upper-triangular matrices with entries in $\bb{Q}$ and $0$'s on the diagonal, with Lie bracket the commutator. Then the usual power series provide mutually inverse functions
\[\begin{tikzcd} \rm{Tr}_n(0,\bb{Q}) \rar[swap]{\exp} & \rm{Tr}_n(1,\bb{Q}) \arrow[shift left=-1ex,swap]{l}{\log},\end{tikzcd}\]
which by nilpotency resp.~unipotency are given by polynomials with rational coefficients in the matrices. Given $U(\bb{Q}) \subseteq \rm{Tr}_n(1,\bb{Q})$ we have $\mathfrak{u} \subseteq \rm{Tr}_n(0,\bb{Q})$, and the mutually inverse functions \eqref{eqn:log-exp} are obtained by restriction. The following notion shall be useful for our arguments; it uses the notation $\log(N) \coloneqq \{\log(n) \mid n \in N\}$.

\begin{definition}
	A subgroup $N \subset U(\bb{Q})$ is a \emph{lattice subgroup} if $\log(N) \subset \mathfrak{u}$ is an additive subgroup. 
\end{definition}

\begin{lemma}\label{lem:everything-lattice-subgroups}Every finitely generated subgroup $N \subset U(\bb{Q})$ has finite index in a lattice subgroup $N' \subset U(\bb{Q})$ satisfying $\rm{span}_{\bb{Q}} \log(N) = \rm{span}_{\bb{Q}} \log(N')$.\end{lemma}

\begin{proof}Firstly, by \cite[Lemma 6.B.2]{Segal} we may assume that $N \subseteq \rm{Tr}_n(1,\bb{Z})$, the subgroup of upper-triangular matrices with entries in $\bb{Z}$ and $1$'s on the diagonal. The construction preceding \cite[Proposition 6.B.1]{Segal} provides for each $d \in \bb{N}_{>0}$ a lattice subgroup 
\[\Gamma_n(d) \subset \rm{Tr}_n(1,\bb{Q}),\] 
given by those upper-triangular matrices $A$ with $1$'s on the diagonal and $(i,j)$th entry $A_{i,j}$ satisfying $c_{j-i} A_{i,j} \in \bb{Z}$ where $c_1 = d$ and $c_i$ for $i \geq 2$ is determined by $c_i = (i!)^{i+1} c_{i-1}^{i^2}$; this is denoted $\Gamma_n(c_1,\ldots,c_{n-1})$ in \cite[Section 6.B]{Segal}. Then \cite[Theorem 6.B.3]{Segal} says that $N'=\Gamma_n(1) \cap \exp(\rm{span}_{\bb{Q}} \log(N))$ is a lattice subgroup containing $N$ as a finite index subgroup. The final statement follows from the inclusions 
\[\log(N) \subset \log(N') = \log(\Gamma_n(1) \cap \exp(\rm{span}_{\bb{Q}} \log(N))) \subset \rm{span}_{\bb{Q}} \log(N).\qedhere\]\end{proof}

We first give a criterion for a subgroup $N \subset U(\bb{Q})$ to be an arithmetic subgroup \cite[Exercise 6.C.13]{Segal}.

\begin{corollary}\label{cor:condition-arithmetic-in-unipotent} A subgroup $N \subset U(\bb{Q})$ is an arithmetic subgroup of $U(\bb{Q})$ if and only if $N$ is finitely generated and $\rm{span}_\bb{Q} \log(N) = \mathfrak{u}$.\end{corollary}

\begin{proof}Arithmetic groups are finitely generated \cite[1.3(1)]{Serre}, and by the previous lemma we may assume that $N = \Gamma_n(1) \cap \exp(\rm{span}_\bb{Q} \log(N))$. 
	
For $\Rightarrow$ we argue by contradiction. If $\rm{span}_\bb{Q} \log(N) \subsetneq \mathfrak{u}$ then there exists an element $u \in \mathfrak{u} \backslash (\rm{span}_\bb{Q} \log(N))$ and replacing it with some multiple if necessary we can assume that $\exp(u) \in U_\bb{Z}$. As $\log(\exp(u)^r) = ru$, we see that all positive powers of $\exp(u)$ are distinct but none lie in $N$, so that $N$ can not have finite index in $U_\bb{Z}$. For $\Leftarrow$ we observe that if $\rm{span}_\bb{Q} \log(N) = \mathfrak{u}$ then $N = \Gamma_n(1) \cap U(\bb{Q})$ and this contains $U_\bb{Z} = \rm{Tr}_n(1,\bb{Z}) \cap U(\bb{Q})$ as a finite index subgroup because $\rm{Tr}_n(1,\bb{Z})$ has finite index in $\Gamma_n(1)$.
\end{proof}

It is a consequence that any finitely generated subgroup $N$ of $U(\bb{Q})$ is arithmetic; more precisely, it is an arithmetic subgroup of $\exp(\rm{span}_\bb{Q} \log(N))$. We next give a technical result which allows us add elements to lattice subgroups with group actions.

\begin{lemma}\label{lem:g-fixed-lattice} Suppose a group $G$ acts on $U(\bb{Q})$. If $N$ is a finitely generated lattice subgroup of $U(\bb{Q})$ preserved setwise by $G$ and satisfying $\rm{span}_\bb{Q}\,\log(N) = \mathfrak{u}$, and $n_1,\ldots,n_r$ are finitely many elements of $U(\bb{Q})$, then there is another lattice subgroup $N'$ of $U(\bb{Q})$ preserved setwise by $G$ and containing both $N$ and $n_1,\ldots,n_r$.\end{lemma}

\begin{proof}We start with an observation. If $L \subset V$ is a lattice in a finite dimensional $\bb{Q}$-vector space $V$, and $v_1,\ldots,v_r$ are finitely many elements of $V$, then there exists an integer $d \in \bb{N}_{>0}$ such that $\frac{1}{d}L$ contains $v_1,\ldots,v_r$. Indeed, with loss of generality $V = \bb{Q}^n$ and $L = \bb{Z}^n$, and we can take $d$ to be the least common multiple of the denominators of entries of $v_1,\ldots,v_r$. If a group $G$ acting on $V$ preserves $L$ setwise, then it also preserves $L'$ setwise. To prove the lemma, we apply this observation to $L=\log(N)$ and $v_i = \log(n_i)$ and take $N'$ to be the smallest lattice subgroup containing $\exp(\frac{1}{d}L)$. This is given by the intersection of all lattice subgroups containing $\exp(\frac{1}{d}L)$, which exists because we may assume without loss of generality that $N = \exp(L) \subseteq \Gamma_n(1) \cap U(\bb{Q})$ and then $\exp(\frac{1}{d}L) \subset \Gamma_n(d) \cap U(\bb{Q})$. Moreover, $N'$ is preserved setwise by $G$ because otherwise $\bigcap_{g \in G} g(N')$ would be a smaller lattice subgroup containing $\exp(\frac{1}{d}L)$.\end{proof}

\subsection{Extensions of arithmetic groups} For a linear algebraic group $G$, the \emph{unipotent radical} $U \subseteq G$ is the largest unipotent subgroup of $G$. Its quotient $R = G/U$ is \emph{reductive}, i.e.~has trivial unipotent radical, and the existence of Levi decompositions \cite[p.\,158]{Borel} gives a splitting $s \colon R \to G$ exhibiting $G$ as a semi-direct product
\[G \cong R \ltimes U\]
as in \cref{lem:semi-direct}. Taking $\bb{Q}$-points, we obtain a split exact sequence
\[\begin{tikzcd} 1 \rar & U(\bb{Q}) \rar & G(\bb{Q}) \rar{\pi} & R(\bb{Q}) \arrow[bend left=15]{l}{s} \rar & 1.\end{tikzcd}\]
We use this to characterise arithmetic subgroups of $G(\bb{Q})$.

\begin{lemma}\label{lem:arith-levi-char} Let $G$ be a linear algebraic group over $\bb{Q}$ with Levi decomposition $G \cong R \ltimes U$ and let $A \subset G(\bb{Q})$ be a subgroup. Then the following are equivalent:
	\begin{enumerate}
		\item $A$ is an arithmetic subgroup.
		\item When we form
		\[\begin{tikzcd} 1 \rar & N \coloneqq A \cap U(\bb{Q}) \rar \dar[hook] & A \rar \dar[hook] & \Gamma \coloneqq \pi(A) \dar[hook] \rar & 1 \\
			1 \rar & U(\bb{Q}) \rar & G(\bb{Q}) \rar{\pi} & R(\bb{Q}) \rar \arrow[bend left=15]{l}{s} & 1  \end{tikzcd}\]
		we have that $N \subset U(\bb{Q})$ and $\Gamma \subset R(\bb{Q})$ are arithmetic subgroups.
	\end{enumerate}
\end{lemma}

\begin{proof}For (1) $\Rightarrow$ (2), the groups $\Gamma = \pi(A)$ and $N = A \cap U(\bb{Q})$ are arithmetic by \cref{lem:arithm-constructions} (i) and (ii) respectively.
	
\smallskip
	
	
For (2) $\Rightarrow$ (1), it suffices to exhibit an arithmetic subgroup of $G(\bb{Q})$ which contains $A$ as a finite index subgroup. Since $\Gamma$ and $N$ are arithmetic they are finitely generated, so $A$ is also finitely generated and we may pick a finite generating set $a_1,\ldots,a_p$ of $A$ and set $u_i = a_is(\pi(a_i))^{-1} \in U(\bb{Q})$. Using \cref{lem:g-fixed-lattice}, let $N'$ be a lattice subgroup of $U(\bb{Q})$ preserved by $\Gamma$ and containing $N$ as well as $u_1,\ldots,u_p$. As $\Gamma$ leaves $N'$ invariant, we can consider the semi-direct product $\Gamma \ltimes N'$. We have arranged that $\Gamma \ltimes N'$ contains $A$: the crucial observation is that if $a = (r,u) \in A \subset R(\bb{Q}) \ltimes U(\bb{Q})$ then $u \in N'$. This is true because $a = (r,u)$ is a product of powers of $a_i = (r_i,u_i)$ and hence $u$ is a product of powers of $u_i$ acted upon by elements of $\Gamma$, which lies $N'$ because it contains the $u_i$'s and is preserved by $\Gamma$. Since there is a map of extensions
\[\begin{tikzcd} 1 \rar & N \rar \dar[hook] & A \rar \dar[hook] & \Gamma \rar \dar[equal] & 1 \\
	1 \rar & N' \rar & \Gamma \ltimes N' \rar & \Gamma \rar & 1 \end{tikzcd}\]
with the images of the outer vertical homomorphisms having finite index, the image of the middle vertical homomorphism also has finite index.
\end{proof}

This has the following consequence, attributed without proof to Borel and Gr\"unewald in \cite{TriantafillouTelAviv}:

\begin{corollary}\label{cor:borel-lemma} Suppose we have a map of short exact sequences of groups
		\[\begin{tikzcd} 1 \rar & A \rar \dar{r} & A' \rar \dar{r'} & A'' \rar \dar{r''} & 1 \\
			1 \rar & G(\bb{Q}) \rar & G'(\bb{Q}) \rar & G''(\bb{Q}) \rar & 1 \end{tikzcd} \]
		with the following properties:
		\begin{enumerate}
			\item The maps $r$ and $r''$ have finite kernel and image an arithmetic subgroup.
			\item The bottom row arises from a short exact sequence of linear algebraic groups with $G$ unipotent.
		\end{enumerate}
		Then $r'$ has finite kernel and image an arithmetic subgroup.
\end{corollary}

\begin{proof}That $r'$ has a finite kernel follows from the exact sequence
	\[1 \lra \ker(r) \lra \ker(r') \lra \ker(r'').\]
	
\smallskip
	
To prove that the image of $r'$ is an arithmetic subgroup, we write $G = U$ to indicate it is unipotent. Picking a Levi decomposition $G' \cong R \ltimes U'$ we get an induced Levi decomposition $G'' \cong R \ltimes U''$ with $U'' = U'/U$. Thus we may replace the bottom row with
	\[1 \lra U(\bb{Q}) \lra R(\bb{Q}) \ltimes U'(\bb{Q}) \lra R(\bb{Q}) \ltimes U''(\bb{Q}) \lra 1\]
	where $U''(\bb{Q}) = U'(\bb{Q})/U(\bb{Q})$. Using the condition that $r''(A'')$ is arithmetic and (1) $\Rightarrow$ (2) of \cref{lem:arith-levi-char}, we see that $\pi_R(r'(A')) = \pi_R(r''(A'')) \subset R(\bb{Q})$ is arithmetic. It remains to prove that $r'(A') \cap U'(\bb{Q})$ is arithmetic, which by \cref{cor:condition-arithmetic-in-unipotent} follows if it is finitely generated and its logarithms span $\mathfrak{u}'$. The first is a consequence of the short exact sequence of groups
	\[1 \lra r(A) \cap U(\bb{Q}) \lra r'(A') \cap U'(\bb{Q}) \lra r''(A'') \cap U''(\bb{Q}) \lra 1\]
	where the outer terms are finitely generated because they are arithmetic subroups as a consequence of \cref{lem:arithm-constructions} \eqref{enum:intersection-arithm}. The second follows from the short exact sequence of $\bb{Q}$-vector spaces
	\[0 \lra \mathfrak{u} \lra \mathfrak{u'} \lra \mathfrak{u}'' \lra 0\]
	where the outer terms are spanned by logarithms of elements of $r(A) \cap U(\bb{Q})$ and $r''(A'') \cap U''(\bb{Q})$ respectively.
\end{proof}

\section{Homotopy automorphisms}

\subsection{Rationalisation of spaces} Rationalisation of spaces is the left adjoint 
\[(-)_\bb{Q} \colon \icat{S} \lra \icat{S}_\bb{Q}\]
in a reflective localisation of $\infty$-categories, where the left side is the $\infty$-category of spaces and the right side is the full sub-category of $\bb{Q}$-local spaces. In particular, the unit of the adjunction gives a natural transformation $r \colon \rm{id}_\icat{S} \to (-)_\bb{Q}$. It may be described in terms of simplicial model categories by a continuous functor \cite[Theorem I.C.13]{Farjoun} which preserves cofibrations by the construction in \cite[Section I.B]{Farjoun}; these results apply to Bousfield homological rationalisation \cite[Section I.E.4]{Farjoun} and agrees on nilpotent spaces of finite type \cite[p.~133]{Bousfield} with the more familiar Bousfield--Kan rationalisation \cite[I.E.3]{Farjoun}.

When applied to a nilpotent pointed space of finite type, it yields once more a nilpotent pointed space of finite type. In this case the effect on homotopy groups well-understood; it is a rationalisation in the sense of \cite{HiltonMilsinRoitberg}. On fundamental groups the homomorphism
\[\pi_1(X) \lra \pi_1(X_\bb{Q})\]
is naturally isomorphic to the Malcev completion; this is implicit in loc.cit.~but spelled out in \cite{CenklPorter}. Thus, by the description of Malcev completion on \cite[p.\,107]{Segal}, up to isomorphism, this homomorphism is equal to the following composition : first we take the quotient $N$ of $\pi_1(X)$ by its torsion subgroup (well-defined by \cite[Corollary I.B.10]{Segal}), then we pick an inclusion $N \subseteq \rm{Tr}_n(1,\bb{Z})$ \cite[Theorem 5.B.2]{Segal}, and finally include $N$ into $\exp(\rm{span}_\bb{Q} \log(N))$ (with notation as in \cref{sec:unipotent}). This yields the following result:

\begin{proposition}\label{prop:rationalising-pi-1} If $X$ is a nilpotent pointed space of finite type, then $\pi_1(X_\bb{Q})$ is isomorphic to the $\bb{Q}$-points of a unipotent linear algebraic group, unique up to isomorphism. Moreover, the homomorphism
\[\pi_1(X) \lra \pi_1(X_\bb{Q})\]
has finite kernel and image given by an arithmetic subgroup.\end{proposition}

One further result about rationalisation that we need, is the following result concerning mapping spaces \cite[Theorem II.2.6, Theorem II.3.11, Corollary II.5.4 (a)]{HiltonMilsinRoitberg}:

\begin{proposition}\label{prop:rationalising-mapping-spaces} If $X$ is a connected nilpotent space of finite type and $W$ is a connected finite CW-complex, the map on sets of path components
	\[\pi_0\, \Map(W,X) \lra \pi_0 \,\Map(W,X_\bb{Q})\]
induced by the rationalisation map $r \colon X \to X_\bb{Q}$ is finite-to-one. Moreover, the path component $\Map(W,X)_f$ of a map $f \colon W \to X$ is nilpotent and the map between path components
	\[\Map(W,X)_f \lra \Map(W,X_\bb{Q})_{rf}\]
is a rationalisation.
\end{proposition}
	
\subsection{Relative homotopy automorphisms} The mapping spaces in the $\infty$-category $\icat{S}$ and the slice category $\icat{S}^{A/}$ can often be computed explicitly, as these $\infty$-categories are obtained from simplicial model categories. 

For topological spaces $X$ and $Y$, $\rm{Map}(X,Y)$ is the topological space of continuous maps $X \to Y$ in the compact open topology; if $X$ and $Y$ are (retracts of) CW-complexes these are the mapping spaces in $\icat{S}$. In $\icat{S}^{A/}$ we need to represent morphisms by a cofibration $A \to X$ between (retracts of) CW-complexes and then the mapping spaces are given by $\rm{Map}^{A/}(X,Y) \subseteq \rm{Map}(X,Y)$, the subspace of continuous maps $X \to Y$ under the identity of $A$. Similar, the following compute automorphism spaces in $\icat{S}$ and $\icat{S}^{A/}$ respectively:

\begin{definition} \qquad
	\begin{enumerate}[(i)]
	\item For a CW-complex $X$, $\rm{hAut}(X) \subseteq \rm{Map}(X,X)$ is the union of the homotopy-invertible path components.
	\item For a cofibration $A \to X$ of CW-complexes, $\rm{hAut}^{A/}(X) \subseteq \rm{Map}^{A/}(X,X)$ is the union of the homotopy-invertible path components.
	\end{enumerate} 
\end{definition}

If $X$ is a $1$-connected finite CW-complex, then $\pi_0\,\hAut(X_\bb{Q})$ is isomorphic to the $\bb{Q}$-points of a linear algebraic group (over $\bb{Q}$, per \cref{conv:q}). Several constructions of such an isomorphism exist: Sullivan \cite{Sullivan} (see \cite{BlockLazarev} for some additional details), Wilkerson \cite{Wilkerson,WilkersonErrata}, and Espic--Saleh \cite{EspicSaleh} (strictly speaking, they require a choice of basepoint but since $X$ is $1$-connected pointed homotopy classes coincide with homotopy classes). A priori, these need not coincide though we expect they do; at least Sullivan's and Espic--Saleh's coincide \cite[Theorem 4.14]{EspicSaleh}. In this paper we shall always consider this structure of a linear algebraic group on $\pi_0\,\hAut(X_\bb{Q})$, which we will recall in \cref{sec:haut-alg}.

The map of fibration sequence
\begin{equation}\label{eqn:fib-seq-haut} \begin{tikzcd} \Map^{A/}(X,X) \rar \dar& \Map(X,X) \rar \dar & \Map(A,X) \dar \\
\Map^{A_\bb{Q}/}(X_\bb{Q},X_\bb{Q}) \rar& \Map(X_\bb{Q},X_\bb{Q}) \rar & \Map(A_\bb{Q},X_\bb{Q})\end{tikzcd}\end{equation}
with fibers taken over $i$ or $i_\bb{Q}$, induce a map of exact sequences 
\[\begin{tikzcd}\cdots \rar & \pi_1\,\rm{Map}(A,X) \rar{\delta} \dar{r} & \pi_0 \rm{hAut}^{A/}(X) \rar{j} \dar{r} & \pi_0 \rm{hAut}(X) \dar \\
\cdots \rar & \pi_1\,\rm{Map}(A_\bb{Q},X_\bb{Q}) \rar{\delta_\bb{Q}} & \pi_0 \rm{hAut}^{A_\bb{Q}/}(X_\bb{Q}) \rar{j_\bb{Q}}  & \pi_0 \rm{hAut}(X_\bb{Q}) \end{tikzcd}\]
based at the identities, $i$, or $i_\bb{Q}$.
 
By \cite[Theorem 4.13]{EspicSaleh}, $\pi_0\, \rm{hAut}^{A_\bb{Q}/}(X_\bb{Q})$ is a linear algebraic group and by \cite[Theorem 4.15]{EspicSaleh} the homomorphism $j_\bb{Q}$ is given by taking $\bb{Q}$-points of a morphism of linear algebraic groups with unipotent kernel. By \cref{lem:alg-grp-constr} \eqref{enum:images-alg} and \cref{lem:unipotent-q-points}, $\rm{im}(j_\bb{Q})$ is given by the $\bb{Q}$-points of a linear algebraic group and by \cite[Proposition 5.4]{EspicSaleh} the homomorphism
\[\rm{im}(j) \lra \rm{im}(j_\bb{Q})\]
has finite kernel and image an arithmetic subgroup of $\rm{im}(j_\bb{Q})$. We use this to prove:

\begin{proposition}If $A \to X$ is a cofibration of $1$-connected finite CW-complexes, then the homomorphism 
	\[\pi_0\,\rm{hAut}^{A/}(X) \lra \pi_0\, \rm{hAut}^{A_\bb{Q}/}(X_\bb{Q})\]
has finite kernel and image an arithmetic subgroup of the linear algebraic group $\pi_0\, \rm{hAut}^{A_\bb{Q}/}(X_\bb{Q})$.\end{proposition}

\begin{proof}There is a map of short exact sequences
	\[\begin{tikzcd} 1 \rar & \ker(j) \rar \dar{r} & \pi_0\,\hAut^{A/}(X) \rar{j} \dar{r'} & \rm{im}(j) \dar{r''} \rar & 1 \\
	1 \rar & \ker(j_\bb{Q}) \rar & \pi_0\,\hAut^{A_\bb{Q}/}(X_\bb{Q}) \rar{j_\bb{Q}} & \rm{im}(j_\bb{Q}) \rar & 1 \end{tikzcd}\]
which satisfies most of the conditions of \cref{cor:borel-lemma} since kernels of morphisms of linear algebraic groups are again linear algebraic groups by \cref{lem:alg-grp-constr} \eqref{enum:kernels-alg}. We will prove  that $r$ has finite kernel and image an arithmetic subgroup, with respect to \emph{some} unipotent algebraic group structure on $\ker(j_\bb{Q})$. Since this linear algebraic group structure must coincide with the one that $\ker(j_\bb{Q})$ inherits as the kernel of $j_\bb{Q}$, because such structures are unique in the unipotent case \cite[Corollary 3.8 (1)]{Chatterjee}, this would complete the proof.

The long exact sequences of homotopy groups of \eqref{eqn:fib-seq-haut} provide a commutative diagram
\[\begin{tikzcd} \pi_1\,\hAut(X) \rar \dar{r_0} & \pi_1 \, \Map(A,X) \rar \dar{r_1} & \ker(j) \rar \dar{r} & 1 \\
\pi_1\,\hAut(X_\bb{Q}) \rar & \pi_1\,\Map(A_\bb{Q},X_\bb{Q}) \rar & \ker(j_\bb{Q}) \rar & 1.\end{tikzcd}\]
As a consequence of \cref{prop:rationalising-mapping-spaces}, the groups $\pi_1\,\hAut(X)$ and $\pi_1\,\Map(A,X)$ are finitely generated nilpotent and the maps $r_0$ and $r_1$ are their rationalisations. Thus $\ker(j)$ is also finitely generated nilpotent and by \cite[Corollary I.2.6]{HiltonMilsinRoitberg} the homomorphism $r$ is its rationalisation. By \cref{prop:rationalising-pi-1} this implies that $\ker(j_\bb{Q})$ admits the structure of a unipotent linear algebraic group, and that $r$ has finite kernel and image an arithmetic subgroup with respect to this structure.
\end{proof}

\begin{remark}That $\pi_0\,\rm{hAut}^{A/}(X)$ is commensurable up to finite kernel with an arithmetic group was claimed in \cite[Section 8]{Scheerer} and that it is of finite type can be deduced from \cite[Theorem 2.2]{Maruyama}. However, these papers provide few details (see the end of the introduction of \cite{EspicSaleh}).\end{remark}

\subsection{Homotopy automorphisms stabilising cohomology classes} In this section we prove:

\begin{proposition}\label{prop:action-algebraic} For a cofibration $A \to X$ of $1$-connected finite CW-complexes, the action of the linear algebraic group $\pi_0\,\hAut^{A_\bb{Q}/}(X_\bb{Q})$ on $H^*(X,A;\bb{Q})$ is algebraic.\end{proposition}

Before doing so, let us draw a consequence. 

\begin{definition}
	For a cofibration $A \to X$ of CW-complexes and classes $h_1,\ldots,h_r$ in $H^*(X,A;\bb{Q})$, $\hAut^{A/}(X)_h \subset \hAut^{A/}(X)$ is the union of the path components fixing each of the $h_i$.
\end{definition}

If $A$ and $X$ are 1-connected of finite type then the map $H^*(X_\bb{Q},A_{\bb{Q}};\bb{Q}) \to H^*(X,A;\bb{Q})$ is an isomorphism. Then we can transfer the $h_1,\ldots,h_r$ to the left term, denoting them the same, and define 
\[\hAut^{A_{\bb{Q}}/}(X_{\bb{Q}})_h \subset \hAut^{A_\bb{Q}/}(X_\bb{Q}).\] 
Its path components are a linear algebraic subgroup of $\pi_0\,\hAut^{A_{\bb{Q}}/}(X_{\bb{Q}})$, since it is a finite intersection of stabilisers of elements in an algebraic representation and we may invoke \cref{lem:alg-grp-constr} \eqref{enum:intersection-alg-subgrp} and \eqref{enum:stabilisers-alg}.

\begin{corollary}\label{cor:haut-fixing-cohomology} If $A \to X$ is a cofibration of $1$-connected finite CW-complexes and $h_1,\ldots,h_r$ are finitely many rational cohomology classes, then the homomorphism
	\[\pi_0\,\hAut^{A/}(X)_h  \lra \pi_0\,\hAut^{A_{\bb{Q}}/}(X_{\bb{Q}})_h\]
has finite kernel and image an arithmetic subgroup.
\end{corollary}

\begin{proof}Its kernel is a subgroup of the kernel of $\pi_0\,\hAut^{A/}(X) \to \pi_0\,\hAut^{A_{\bb{Q}}/}(X_{\bb{Q}})$. Its image is the intersection of the image of $\pi_0\,\hAut^{A/}(X)$ in $\pi_0\,\hAut^{A_{\bb{Q}}/}(X_{\bb{Q}})$ with $\pi_0\,\hAut^{A_{\bb{Q}}/}(X_{\bb{Q}})_h$, and we may invoke \cref{lem:arithm-constructions} \eqref{enum:intersection-arithm}.\end{proof}

It will take some effort to prove \cref{prop:action-algebraic}, as we must discuss the isomorphism of $\pi_0 \,\hAut^{A_\bb{Q}/}(X_\bb{Q})$ with the $\bb{Q}$-points of a linear algebraic group. For the expert, however, we can summarise the argument in a single unsurprising sentence: the linear algebraic group of automorphisms of a minimal Quillen model acts algebraically on its Chevalley--Eilenberg cohomology.

\subsubsection{Recollection of rational homotopy theory} To prove \cref{prop:action-algebraic} we first recall the basics of rational homotopy theory. It concerns the commuting diagram of equivalences of $\infty$-categories
\begin{equation}\label{eqn:rat-hmtp}\begin{tikzcd} & \icat{S}^{\fin,>1}_{\ast,\bb{Q}} \arrow{rd}{A}[swap]{\simeq} \arrow{ld}{\simeq}[swap]{\cal{L}}& \\
\dgLie^{\fin,>0} \arrow[shift left=.5ex]{rr}{C^*_\rm{CE}} & & \dgCom^{\fin,>1}_{\aug} \arrow[shift left=.5ex]{ll}{C^*_\rm{Har}}\end{tikzcd}\end{equation}
Here the notation is as follows:
\begin{enumerate}
	\item $ \icat{S}^{\fin,>1}_{\ast,\bb{Q}}$ is the $\infty$-category of $1$-connected pointed rational spaces of finite type.
	\item $\dgLie^{\fin,>0}$ is the $\infty$-category of $0$-connected dg-Lie algebras of finite type.
	\item $\dgCom^{\fin,>1}_{\aug}$ is the $\infty$-category of $1$-connected augmented dg-commutative algebras of finite type.
	\item $\cal{L}$ is the functor denoted $\lambda$ in \cite{QuillenRational}.
	\item $C^*_\rm{CE}$ is the Lie algebra cohomology functor and $C^*_\rm{Har}$ is the commutative algebra cohomology functor.
	\item $A$ is the functor denoted $A^*_\rm{PL}$ in \cite{Sullivan}. 
\end{enumerate}

To justify this, we need some results from the literature phrased in terms of model categories:

\begin{enumerate}
	\item In \cite[Section II.2]{QuillenRational}, Quillen constructs a model structure on the category $\cat{S}_2$ of $1$-reduced simplicial sets with weak equivalences those maps inducing an isomorphism on rational homology; there is an equivalence of $\infty$-categories $\smash{\icat{S}^{>1}_{\ast,\bb{Q}}} \simeq (\smash{\cat{S}_2^\rm{Q}})^\circ$ where $(-)^\circ$ denote the (coherent nerve of) the $\cat{Kan}$-enriched category of cofibrant-fibrant objects as in \cite{LurieHTT}, and we can restrict to the full subcategory of objects of finite type.
	\item In \cite[Section II.5]{QuillenRational}, Quillen constructed a model structure on the category $\cat{DGL}_1$ of reduced dg-Lie algebras with weak equivalences those maps inducing an isomorphism on homology; there is an equivalence of $\infty$-categories $\dgLie^{>0} \simeq (\cat{DGL}_1^\rm{Q})^\circ$, and we can restrict to the full subcategory of objects of finite type.
	\item In \cite[\S 4]{BousfieldGugenheim}, Bousfield and Gugenheim constructed a model structure on the category $\cat{A}_0$ of  augmented dg-commutative algebras with weak equivalences those maps that induce an isomorphism on homology; there is an equivalence of $\infty$-categories $\dgCom_\aug \simeq (\cat{A}^\rm{BK}_0)^\circ$, and we can restrict to the full subcategory of 1-connected objects of finite type.
	\item In \cite[Theorem II]{QuillenRational}, Quillen exhibits $\lambda$ as arising from a zigzag of Quillen equivalences between $\cat{S}_2^\rm{Q}$ and $\smash{\cat{DGL}_1^\rm{Q}}$. Upon restriction to appreciate subcategories, this yields an equivalence of $\infty$-categories $\icat{S}^{>1}_{\ast,\bb{Q}} \simeq \dgLie^{\fin,>0}$.
	\item The functors $C^*_\rm{CE}$ and $C^*_\rm{Har}$ can described explicitly in terms of the Chevalley--Eilenberg and Harrison cochain complexes, and that they are mutually inverse equivalences is essentially contained in \cite[Appendix B, Theorem 7.5]{QuillenRational}. The difference is that Quillen does not dualise, but rather works with coalgebras. This is necessary in the general case, but under finite type assumptions dualisation is an equivalence.
	\item Bousfield and Gugenheim exhibited $A^*_\rm{PL}$ as a part of a Quillen adjunction between $\smash{\cat{S}^\rm{Q}_\ast}$, the model category of pointed simplicial sets with the Quillen model structure, and $\cat{A}^\rm{BK}_0$ \cite[\S8]{BousfieldGugenheim}. They also proved it becomes a Quillen equivalence when restricted to 1-connected objects of finite type \cite[\S9]{BousfieldGugenheim}, and hence it induces an equivalence of $\infty$-categories $\icat{S}^{>1}_{\ast,\bb{Q}} \simeq \dgCom_\aug^{\fin,>1}$.
	\item The commutativity of the diagram is the one of the main results of \cite{Majewski}. To see this, consider the chain of equivalences of dg-commutative algebras on page xvii of loc.cit.: it connects the universal enveloping algebra $U\cal{L}(X)$ (Majewski's $\lambda G$ is Quillen's $\lambda$) to a variant $F^*A(X)$ of the linear dual of the bar construction on $A(X)$. This gives the desired result upon applying primitives, after noting that (a) by \cite[Lemma 4.43]{Majewski} for finite commutative dg-algebras the construction $F^*$ agrees up to equivalence with the linear dual of the bar construction, (b) $P$ is inverse to $U$ up to equivalence \cite[Theorem 4.8]{QuillenRational}, (c) the Harrison cochain complex can be obtained as primitives in the dual of the bar construction by noting that under finite type assumptions this is equivalent to the cobar construction of the dual and then we can apply \cite[Appendix B, Remark 6.6]{QuillenRational}. 
\end{enumerate}

\subsubsection{Homotopy automorphisms as a linear algebraic group} \label{sec:haut-alg} Let us first recall why $\pi_0\,\hAut_*(X)$ admits the structure of a linear algebraic group for $X \in \icat{S}^{\fin,>1}_{\ast,\bb{Q}}$. 

In any $\infty$-category $\icat{C}$, equivalences induce isomorphisms in the homotopy category $\rm{ho}\,\icat{C}$. Since isomorphic objects have isomorphic automorphism groups, we conclude that an equivalence $X \simeq X'$ yields an isomorphism of groups
\[\pi_0\,\hAut_\icat{C}(X) \cong \pi_0\,\hAut_\icat{C}(X'),\]
where $\hAut_\icat{C}(X) \subseteq \Map_\icat{C}(X,X)$ denotes the union of the path components invertible under composition, and similarly for $X'$.

To compute these automorphism groups in $\dgLie^{\fin,>0}$, one uses \emph{minimal} dg-Lie algebras as in \cite[Definition, p.~311]{FHT}. A minimal dg-Lie algebra is of the form $(\bb{L}(U),d)$ where $\bb{L}(U)$ is a free graded Lie algebra on a graded vector space of finite type concentrated in positive degrees, and differential $d$ has zero linear part. Every dg-Lie algebra $L \in \dgLie^{\fin,>0}$ admits a minimal model, by definition a quasi-isomorphism $M \to L$ with $M$ minimal. In particular, there is an isomorphism of groups
\[\pi_0\, \hAut_\dgLie(L) \cong \pi_0\,\hAut_\dgLie(M).\]
The crucial property of minimal dg-Lie algebras is then that the group $\pi_0\,\hAut_\dgLie(M)$ can be computed explicitly:
\[\pi_0\,\hAut_\dgLie(M) = \Aut(M)/\{\text{homotopic to $\rm{id}_{M}$}\},\]
the quotient of the group $\Aut(M)$ of automorphisms of $M$ as a dg-Lie algebra, by the subgroup of automorphisms homotopic to the identity. To see this, one first uses that minimal dg-Lie algebras are both cofibrant and fibrant (in fact, all objects are fibrant) in Quillen's model structure on dg-Lie algebras, so their $\infty$-categorical automorphism spaces agree with the model-categorical simplicial automorphism spaces. Next one can identify the path components of the model-categorical simplicial automorphism spaces using the techniques of \cite{EspicSaleh}, where this is not explicitly stated but its proof follows by the same methods (see in particular Corollary 4.8 of loc.cit.).

If the graded vector space $U$ of generators of $M$ is finite-dimensional, then the group $\Aut(M)$ is isomorphic to the $\bb{Q}$-points of a linear algebraic group, and under this isomorphism the subgroup of automorphisms homotopic to the identity is given by the $\bb{Q}$-points of a unipotent closed subgroup. Using \cref{lem:unipotent-q-points}, we conclude that $\pi_0\,\hAut_\dgLie(M)$ is isomorphic to the $\bb{Q}$-points of a linear algebraic group as well. Moreover, by construction its action on $M$ is algebraic.

\medskip

To see that $\pi_0\,\hAut_*(X)$ is a linear algebraic group for $X \in \icat{S}^{\fin,>1}_{\ast,\bb{Q}}$, we choose a minimal model $M \to \cal{L}(X)$ and by the above procedure obtain a pair of isomorphisms of groups
\[\Aut(M)/\{\text{homotopic to $\rm{id}$}\} \cong \pi_0\,\hAut_\dgLie(\cal{L}(X)) \cong \pi_0\,\hAut_*(X),\]
with left one as explained above, and right one a consequence of \eqref{eqn:rat-hmtp}.

\subsubsection{Relative homotopy automorphisms as a linear algebraic group}
We can similarly study morphisms $i \colon A \to X$ in $\icat{C}$: we define $\smash{\Map_\icat{C}^{A/}}(X,X)$ as the mapping space of $i$ in the slice $\infty$-category $\icat{C}^{A/}$ and let 
\[\hAut_\icat{C}^{A/}(X) \subseteq\smash{\Map_\icat{C}^{A/}}(X,X)\]
denote the union of the invertible path components.  Then an equivalence $i \simeq i'$ in $\icat{C}^{A/}$ induces an isomorphism of groups
\[\pi_0\,\hAut^{A/}_\icat{C}(X,X) \cong \pi_0\,\hAut^{A'/}_\icat{C}(X').\]

To compute automorphism groups of morphisms in $\dgLie$, one uses \emph{relative minimal} dg-Lie algebras, as explained in \cite[Section 3, 4]{EspicSaleh}. They prove that every morphism $g \colon L \to K$ in $\dgLie^{\fin,>0}$ admits a relative minimal model, given by a commuting diagram
\[\begin{tikzcd} M = (\bb{L}(U),d) \rar{\simeq} \dar & L \dar{g} \\
	N = (\bb{L}(U \oplus V),d) \rar{\simeq} & K \end{tikzcd}\]
with top map a minimal model, $V$ a finite dimensional graded vector space, and bottom differential satisfying some conditions that we do not need to spell out. These are the cofibrant and fibrant objects in the slice model structure obtained from Quillen's model structure, and arguing as in the absolute case there is an isomorphism
\[\pi_0\,\hAut_\dgLie^{M/}(N) \cong \Aut_{M}(N)/\{\text{homotopic to $\rm{id}_{N}$ rel $M$}\},\]
the quotient of the group $\Aut_{M}(N)$ of automorphisms of $N$ as a dg-Lie algebra fixing $M$ pointwise, by the subgroup of such isomorphisms homotopic to the identity rel $M$. Once more, $\Aut_{M}(N)$ is isomorphic to the $\bb{Q}$-points of a linear algebraic group \cite[Proposition 4.10]{EspicSaleh}, and under this isomorphism the subgroup of automorphisms homotopic to the identity rel $\rm{id}_{M}$ is given by $\bb{Q}$-points of a unipotent closed subgroup \cite[proof of Theorem 4.13]{EspicSaleh}. Moreover, the action of this group on the morphism $M \to N$ is by the identity on $M$ and induced by the action of $\Aut(N)$ on $N$; both are algebraic actions. One concludes that $\pi_0\,\smash{\hAut_\dgLie^{M/}}(N)$ is given by the $\bb{Q}$-points of a linear algebraic group.

\medskip

To see that $\pi_0\,\hAut^{A/}(X)$ is given by the $\bb{Q}$-points of a linear algebraic group for a morphism $f \colon A \to X$ in $\icat{S}^{\fin,>1}_{\ast,\bb{Q}}$, we represent it by a cofibration and choose a minimal model $M \to N$ for $\cal{L}(A) \to \cal{L}(X)$. By the above procedure we obtain a pair of isomorphisms of groups
\[\Aut_M(N)/\{\text{homotopic to $\rm{id}_N$ rel $M$}\} \cong \pi_0\,\hAut^{\cal{L}(A)/}_\dgLie(\cal{L}(X)) \cong \pi_0\,\hAut^{A/}(X).\]

\subsubsection{Action on cohomology} There is a forgetful functor of $\infty$-categories
\[U \colon \dgCom^{\fin,>1}_\aug \lra \Ch_\bb{Q}\] sending a dg-commutative algebra to its underlying cochain complex of $\bb{Q}$-vector spaces (cochain complexes are identified with chain complexes by negation of grading). When there is no risk of confusion, we will use the same notation for a dg-commutative algebra and its underlying cochain complex. In particular, given a map $g \colon L \to K$ of dg-Lie algebras, we write 
\[C^*_\rm{CE}(K,L) \coloneqq \rm{fib}_{\Ch_\bb{Q}}(C^*_\rm{CE}(K) \to C^*_\rm{CE}(L)).\]
Note the notation neglects to reflect that this depend on $g$. 


\begin{warning}One is tempted to call $H^*_\rm{CE}(g)$ ``relative Lie algebra cohomology'', but it is \emph{not} what is referred to as relative Lie algebra cohomology in e.g.~\cite[Section 22]{ChevalleyEilenberg}.
\end{warning}

Since $U \circ C^*_\rm{CE}$ is a functor, for a morphism $L \to K$ we get an action 
	\[\pi_0\,\hAut^{L/}_{\dgLie}(K) \times H^*_\rm{CE}(K,L) \lra H^*_\rm{CE}(K,L).\]
If the morphism is a minimal relative dg-Lie algebra $M \to N$, we may identify the group acting with $\Aut_M(N)/\{\text{homotopic to $\rm{id}_N$ rel $M$}\}$ and we may ask whether this action is algebraic. This is indeed the case:   

\begin{lemma}If $M \to N$ is a minimal relative dg-Lie algebra, then the action of the group $\Aut_M(N)/\{\text{homotopic to $\rm{id}_N$ rel $M$}\}$ on $H^*_\rm{CE}(N,M)$ is algebraic.\end{lemma}

\begin{proof}First, we observe that by the universal property of quotients it suffices to prove that action of $\Aut_M(N)$ is algebraic. We will verify this using the explicit construction of the Chevalley--Eilenberg cochain complex \cite[Section 14]{ChevalleyEilenberg}. Since this  construction is natural, if a dg-Lie algebra or morphism of dg-Lie algebra comes with an action then the entries of the cochain complex are representations and the differential is equivariant. Since algebraic representations are closed under images, kernels, and quotients by \cref{lem:alg-act-constructions} \eqref{enum:alg-rep-sub-quot}, it suffices to prove that the entries of $C_\rm{CE}^*(N,M)$ are algebraic representations. We recall that in degree $p$ the Chevalley--Eilenberg cochain complex of $N$ is given by
	\[C^p_\rm{CE}(N) = \Hom_\bb{Q}(\Lambda^p N,\bb{Q}).\]
	Its differential will not matter for this argument. Then $C^*_\rm{CE}(N,M)$ is constructed as a shift of the mapping cone of the cochain map 
	\[C^*_\rm{CE}(N) \lra C^*_\rm{CE}(M)\]
	induced by restriction.	The action of $\Aut_M(N)$ on this cochain map is the evident one: on $C^*_\rm{CE}(N)$ it acts via the inclusion $\Aut_M(N) \subseteq \Aut(N)$ through its action on $N$, and on $C^*_\rm{CE}(M)$ it is the identity. The latter is evidently algebraic. For the former, we need that $\Aut(N)$ acts algebraically on $N$, and that exterior and symmetric powers preserve algebraic representations, as well as linear duals, by \cref{lem:alg-act-constructions} \eqref{enum:alg-rep-tensor} and \eqref{enum:alg-rep-dual}.
\end{proof}

This implies that $\pi_0\,\hAut^{A/}(X)$ acts algebraically on $H^*(X,A)$, as the equivalence of functors $A \simeq C^*_\rm{CE} \circ \cal{L}$ shows that the rational cohomology of a space $X \in \cat{S}^{\fin,>1}_{\ast,\bb{Q}}$ is naturally isomorphic to the Chevalley--Eilenberg cohomology of $\cal{L}(X)$. Since relative cohomology of spaces is also constructed through a mapping cone of cochain complexes, we similarly get a natural identification
\[H^*_\rm{CE}(\cal{L}(X),\cal{L}(A)) \cong H^*(X,A)\]
for a morphism $f \colon A \to X$. This completes the proof of \cref{prop:action-algebraic}.

\section{Homeomorphisms} To deduce \cref{athm:main} from \cref{cor:haut-fixing-cohomology} is standard, and we will outline it here.

\subsection{Cerf's theorem and surgery theory} We go from homotopy automorphisms to homeomorphisms through surgery theory and pseudoisotopy theory. We avoid repeating the definitions, just stating the necessary results. 

First, we have a result of Cerf \cite{Cerf}, true even in dimension $5$, see \cite[p.\,11]{HatcherWagoner} for a reference in the case of manifolds with boundary. These results are stated in terms diffeomorphisms, but remain true for homeomorphisms \cite[p.\,7]{HatcherConcordance}, and say that the space of concordance diffeomorphisms is path-connected. This proves the injectivity of the homomorphism in the next theorem, surjectivity being true by the definition of pseudoisotopy classes.

\begin{theorem}\label{thm:cerf} If $M$ is a $1$-connected compact topological manifold of dimension $d \geq 5$, then the natural homomorphism
	\[\pi_0\,\Homeo_\partial(M) \lra \widetilde{\pi}_0\,\Homeo_\partial(M)\]
is an isomorphism, where $\widetilde{\pi}_0\,\Homeo_\partial(M)$ denotes the group of pseudoisotopy classes of homeomorphisms of $M$ fixing a neighbourhood of $\partial M$ pointwise.\end{theorem}

For $M$ as in \cref{thm:cerf} with $\partial M \neq \varnothing$, fix a disc $D^{d-1} \subset \partial M$ and write $\partial_0 M \coloneqq \partial M \setminus \rm{int}(D^{d-1})$. The inclusion induces a homomorphism \[\widetilde{\pi}_0\,\Homeo_\partial(M) \lra \widetilde{\pi}_0 \,\Homeo_{\partial_0 M}(M)\]
to the group of pseudoisotopy classes of homeomorphisms of $M$ fixing $\partial_0 M$ pointwise, which is an isomorphism by the Alexander trick. The right term can be studied by surgery theory; good references are \cite{Wall,Nicas} and see \cite[Section 2]{Krannich} for an overview. For a reference in the topological category in low dimensions, see \cite[Section 13.3]{FreedmanQuinn}. 

\begin{theorem}\label{thm:surgery} Let $M$ be as in \cref{thm:cerf}.
	\begin{enumerate}
		\item \label{enum:surgery-left} There is an exact sequence
		\[\pi_1\,\Map_*(M/\partial_0 M,\rm{BTop}) \lra \widetilde{\pi}_0\,\Homeo_{\partial_0}(M) \overset{\upsilon}\lra \pi_0\,\hAut_{\partial_0}(M).\]
		\item \label{enum:surgery-right} The image of $\upsilon$ has finite index in the subgroup of elements that fix the homotopy class of a classifying map $\tau^s_M \colon M \to B\rm{Top}$ relative to $\partial_0 M$.
	\end{enumerate}
\end{theorem}

\begin{proof}A reference for \eqref{enum:surgery-left} is the proof of \cite[Theorem 2.2]{Krannich}, in particular in (21) of loc.cit.~and the identification on the top of page 16. One needs to use that the map $\hAut_{\partial}(M) \to \hAut_{\partial_0}(M,D^{d-1})$ to the homotopy automorphisms fixing $\partial_0 M$ pointwise and preserving $D^{d-1}$ setwise, is a weak equivalence, as are inclusions of homotopy automorphisms into block homotopy automorphisms.
	
\smallskip

For \eqref{enum:surgery-right} we recall some facts from surgery theory. Let $\cal{S}(M;\partial_0 M,D^{d-1})$ denote the \emph{structure set} of the triad $(M,\partial_0 M,D^{d-1})$: it consists of equivalence classes of (simple) homotopy equivalences of triads $\phi \colon (W;\partial_0 W,\partial_1 W) \to (M;\partial M \setminus \rm{int}(D^{d-1}),D^{d-1})$ that restrict to a homeomorphism on $\partial_0 W \to \partial M \setminus \rm{int}(D^{d-1})$, and two such homotopy equivalences $\phi$ and $\phi'$ are equivalent if and only if there is a homeomorphism $f$ such that there is a homotopy $\phi' \circ f \sim \phi$ relative to $\partial M \setminus \rm{int}(D^{d-1})$. 

The group $\pi_0\,\hAut_\partial(M)$ acts on the structure set by precomposition, and image of $\upsilon$ consists of the stabiliser of the identity, which we regard at the basepoint of the structure set. By the $\pi$-$\pi$-theorem the normal invariant map
\[\cal{S}(M;\partial_0 M,D^{d-1}) \lra \cal{N}(M;\partial_0 M,D^{d-1})\]
is a bijection of pointed sets. This target is known as the \emph{set of normal invariants} and can be studied using maps
\[\begin{tikzcd}\cal{N}(M;\partial_0 M,D^{d-1}) \dar{\simeq} & \pi_0\,\Map_{\partial_0}(M,B\rm{Top}) \dar{\simeq} \\[-5pt]
	\pi_0\,\Map_*(M/\partial_0 M,G/\rm{Top}) \rar & \pi_0\,\Map_*(M/\partial_0 M,B\rm{Top}), \end{tikzcd}\]
as in the middle of page 16 of loc.cit. Under these identifications the basepoint of the set $\cal{N}(M;\partial_0 M,D^{d-1})$ is mapped to the path component of $\Map_{\partial_0}(M,B\rm{Top})$ of a classifying map $\tau^s_M$ of the stable tangent bundle $T^s M$ relative to $\partial_0 M$, and the action of $\pi_0\,\hAut_\partial(M)$ is on $\pi_0\,\Map_{\partial_0}(M,B\rm{Top})$ is by precomposition.

The map $G/\rm{Top} \to B\rm{Top}$ is one between infinite loop spaces of finite type, and since the homotopy groups of $G$ are finite it induces an equivalence $(G/\rm{Top})_\bb{Q} \to B\rm{Top}_\bb{Q}$. As a consequence of \cref{prop:rationalising-mapping-spaces}, the map  $\pi_0\,\Map_*(M/\partial_0 M,G/\rm{Top}) \to \pi_0\,\Map_*(M/\partial_0 M,B\rm{Top})$ at the bottom of the diagram is finite-to-one. This finishes the proof of \eqref{enum:surgery-right}.
\end{proof}

\subsection{The proof of \cref{athm:main}} The class of groups of finite type is closed under the following constructions, by \cite[Proposition 2.5]{DrorDwyerKan}:

\begin{lemma}\label{lem:finite-type-constr} The following constructions yield groups of finite type:
	\begin{enumerate}[\noindent (1)]
		\item \label{enum:ft-finite-index} Passing to finite index subgroups.
		\item \label{enum:ft-extensions} Extensions of groups of finite type by groups of finite type.
		\item \label{enum:ft-quotients} Quotients groups of finite type by groups of finite type.
\end{enumerate}
\end{lemma}

Moreover, it includes finite groups, finitely generated nilpotent groups, and arithmetic groups by \cite[Examples 2.7, 2.8]{DrorDwyerKan}. The first of these, combined with all parts of \cref{lem:finite-type-constr}, implies that a group that is commensurable up to finite kernel to a group of finite type, is also of finite type. We first prove the following special case of \cref{athm:main} (or rather, \cref{rem:mainthm} \eqref{enum:stronger}) where the boundary is assumed path-connected:

\begin{theorem}\label{thm:ctd-case} If $M$ is a compact topological manifold of dimension $\geq 5$ and $M$ as well as $\partial M$ is $1$-connected, then $\pi_0\, \Homeo_\partial(M)$ is commensurable up to finite kernel to an extension of an arithmetic group by a finitely generated abelian group. In particular, it is of finite type.\end{theorem}

\begin{proof}By \cref{thm:surgery}, it suffices to prove that \begin{enumerate}[(i)]
		\item \label{enum:proof-left} $\pi_1\,\Map_*(M/\partial_0 M,B\rm{Top})$ is a finitely generated abelian group and 
		\item \label{enum:proof-right} the subgroup $\pi_0\,\hAut_\partial(M)_{\tau^s_M} \subseteq \pi_0\,\hAut_\partial(M)$ of homotopy classes of homotopy automorphisms that fix $\tau^s_M$ is an arithmetic group.
\end{enumerate}
For \eqref{enum:proof-left}, we use that $B\rm{Top}$ is an infinite loop space of finite type and so by \cref{prop:rationalising-mapping-spaces} each path component of $\Map_*(M/\partial_0 M,B\rm{Top})$ is nilpotent of finite type. This proves that $\pi_1\,\Map_*(M/\partial_0 M,B\rm{Top})$ is a finitely generated nilpotent group. Moreover, since each path component of $\Map_*(M/\partial_0 M,B\rm{Top})$ is an H-space, this must in fact be abelian.

For \eqref{enum:proof-right}, we use that composition with the rationalisation $r \colon B\rm{Top} \to B\rm{Top}_\bb{Q}$ yields a map 
\[\pi_0\,\Map_{\partial_0 M}(M,B\rm{Top}) \lra \pi_0\,\Map_{\partial_0 M}(M,B\rm{Top}_\bb{Q})\]
which is finite-to-one as a consequence of \cref{prop:rationalising-mapping-spaces}. Thus $\pi_0\, \hAut_\partial(M)_{\tau^s_M}$ has finite index in $\pi_0\, \hAut_\partial(M)_{r \tau^s_M}$. Now we use that there is a homotopy equivalence
\[B\rm{Top}_\bb{Q} \overset{\simeq}\lra \prod_{i=1}^\infty K(\bb{Q},4i)\] 
given by the rational Pontryagin classes. Thus an element of $\pi_0\, \hAut_\partial(M)$ stabilises the map $r\tau^s_M$ if and only if it stabilises the pullback of each $p_i$ to $H^{4i}(M,\partial_0 M;\bb{Q})$. Since the inclusion induces an isomorphism $H^*(M,\partial M;\bb{Q}) \cong H^*(M,\partial_0M;\bb{Q})$ in degrees $*<d$, while the right term vanishes for $* \geq d$, this is equivalent to stabilising $p_i(T^s M) \in H^{4i}(M,\partial M;\bb{Q})$ for all $4i<d$. At this point we invoke \cref{cor:haut-fixing-cohomology} with cofibration $A \to X$ given by the inclusion $\partial M \to M$ and rational cohomology classes $h_1,\ldots,h_r$ given by the Pontryagin classes $p_i(T^s M)$ for $4i<d$.

To conclude $\pi_0\, \Homeo_\partial(M)$ is of finite type, use the remarks preceding this proofs and \cref{lem:finite-type-constr} \eqref{enum:ft-extensions}.
\end{proof}

To prove \cref{athm:main} and later \cref{cor:nontrivial-pi1}, we need some geometric tools for topological manifolds that are well-known for smooth manifolds: isotopy extension \cite[Theorem 6.17]{Siebenmann}, immersion theory \cite{Lees}, general position \cite{Dancis}, the existence of normal microbundles for submanifolds of dimension $\leq 1$ \cite[Section 9.3, 9.4]{FreedmanQuinn}, and that $\pi_i \, \rm{Top}(d)$ is finite for $i \leq 2$ and finitely generated for $i=3$ \cite[Section 8.3, 8.7]{FreedmanQuinn}. 

\begin{proof}[Proof of \cref{athm:main}] The argument is an induction over the number $s$ of path components of $\partial M$, which is finite because $M$ is compact. The initial case $s=1$ is \cref{thm:ctd-case}. For the induction step for $s$ to $s+1$ we will perform a surgery on the boundary. We fix discs $D^{d-1} \times \{0\}$ and $D^{d-1} \times \{1\}$ in two different path components of $\partial M$ and consider the space $\rm{Emb}^\rm{Top}_\partial(D^{d-1} \times I,M)$ of topological embeddings restricting to the fixed discs on $D^{d-1} \times \{0,1\}$. By general position, the map
	\[\rm{Emb}^\rm{Top}_\partial(D^{d-1} \times I,M) \lra \rm{Imm}^\rm{Top}_\partial(D^{d-1} \times I,M)\]
is $2$-connected when $d \geq 5$. By immersion theory, the latter fits in a fibration sequence
\[\Omega \rm{Top}(d) \lra \rm{Imm}^\rm{Top}_\partial(D^{d-1} \times I,M) \lra \Map_\partial(I,M).\]
Both outer terms are non-empty with finitely many path components: the left term because $\pi_i\,\rm{Top}(d)$ is finite for $i \leq 2$ when $d \geq 6$ and the right term because $M$ is 1-connected. Hence the same is true for the middle term. Similarly, $\pi_1\, \rm{Imm}^\rm{Top}_\partial(D^{d-1} \times I,M)$ is a finitely generated nilpotent group. Picking an embedding $e \colon D^{d-1} \times I \to M$, we obtain a fibration sequence
\[\Homeo_\partial(M') \lra \Homeo_\partial(M) \lra \rm{Emb}^\rm{Top}_\partial(D^{d-1} \times I,M)\]
with $M' \coloneqq M \setminus \rm{int}(\rm{im}(e))$. This implies that a finite index subgroup of $\pi_0\,\Homeo_\partial(M)$ is isomorphic to the quotient of $\pi_0\,\Homeo_\partial(M')$ by the image of $\pi_1 \,\rm{Emb}^\rm{Top}_\partial(D^{d-1} \times I,M)$. On the one hand, since the boundary of $M'$ has $s$ path components, by the induction hypothesis $\pi_0\,\Homeo_\partial(M')$ is of finite type. On the other hand, any quotient of $\pi_1 \,\rm{Emb}^\rm{Top}_\partial(D^{d-1} \times I,M)$ is a finitely generated nilpotent group and hence of finite type. By \cref{lem:finite-type-constr} \eqref{enum:ft-quotients}, the quotient is of finite type, and by \cref{lem:finite-type-constr} \eqref{enum:ft-finite-index} and \eqref{enum:ft-extensions} so is $\pi_0\,\Homeo_\partial(M)$.
\end{proof}

\subsection{Proofs of some remarks}\label{sec:remarks}
Next, we make precise \cref{rem:mainthm} \eqref{enum:connectivity}, which said that in \cref{athm:main} we could drop the assumption that the path components of $\partial M$ are $1$-connected at the expense of additional conditions on $M$.

\begin{corollary}\label{cor:nontrivial-pi1} If $M$ is a compact topological manifold of dimension $d \geq 7$, $M$ is $1$-connected, and $\pi_2\,M$ is finite, then $\pi_0\,\Homeo_\partial(M)$ is of finite type.
\end{corollary}

\begin{proof}We prove this in the case that $\partial M$ is $0$-connected; the general case is done exactly as in the proof of \cref{athm:main} by induction over the number of path components. 
	
Since $\partial M$ is compact, its fundamental group is finitely generated. We do an induction over the number $s$ of generators, the initial case being a special case of \cref{thm:ctd-case}. The strategy for the induction step from $s$ to $s+1$ is similar to the proof of \cref{athm:main}, but using discs to perform a surgery on a thickened loop rather than intervals to perform a surgery on two discs.
	
We represent a generator by an embedding $\gamma \colon S^1 \to \partial M$, which exists by general position. This has a normal microbundle, which must be trivial because $\gamma$ is null-homotopic, so we can extend it to an embedding $\overline{\gamma} \colon D^{d-2} \times S^1 \to \partial M$.	By general position, the map
	\[\rm{Emb}_\partial^\rm{Top}(D^{d-2} \times D^2,M) \lra \rm{Imm}_\partial^\rm{Top}(D^{d-2} \times D^2,M)\]
where the boundary conditions are given by $\overline{\gamma}$, is $2$-connected when $d \geq 7$. As above, there is a fibration sequence
\[\Omega^2 \rm{Top}(d) \lra \rm{Imm}_\partial^\rm{Top}(D^{d-2} \times D^2,M) \lra \rm{Map}_{\partial D^2}(D^2,M).\]
Both outer terms are non-empty with finitely many path components: the left term as above and the right term because the set of null-homotopies of $\gamma$ is a torsor for $\pi_2\,M$. Similarly, $\pi_1 \, \rm{Imm}_\partial^\rm{Top}(D^{d-2} \times D^2,M)$ is a finitely generated nilpotent group. Picking an embedding $e \colon D^{d-2} \times D^2 \to M$ extending $\overline{\gamma}$, we obtain a fibration sequence
\[\Homeo_\partial(M') \lra \Homeo_\partial(M) \lra \rm{Emb}^\rm{Top}_\partial(D^{d-2} \times D^2,M)\]
with $M' \coloneqq M \setminus \rm{int}(\rm{im}(e))$. Observing that $M'$ also satisfies the hypothesis of this corollary but $\partial M'$ has one less generator, the induction hypothesis yields that $\pi_0\,\Homeo_\partial(M')$ is of finite type. Arguing as in the proof of \cref{athm:main}, we conclude from this that $\pi_0\,\Homeo_\partial(M)$ is also of finite type.
\end{proof}

Next, we explain \cref{rem:mainthm} \eqref{enum:smooth}, which claimed that \cref{athm:main} also holds for diffeomorphisms and PL homeomorphisms. If we suppose that $M$ admits a smooth structure, then using smoothing theory as in the proof of \cite[Proposition 2.5]{BustamanteKrannichKupers} one proves that in dimension $\geq 5$ the homomorphism
\[\pi_0\,\Diff_\partial(M) \lra \pi_0\,\Homeo_\partial(M)\]
has finite kernel and its image has finite index. Thus the result in \cref{athm:main} for homeomorphisms implies the result for diffeomorphisms, as does that in \cref{cor:nontrivial-pi1}. A similar argument, using triangulation theory rather than smoothing theory, works for PL homeomorphisms of a PL manifold.

\begin{remark}General position results are more refined for smooth manifolds than for topological manifolds, which allows one to extend \cref{cor:nontrivial-pi1} to $d \geq 6$ for diffeomorphisms. We leave the details to the interested reader.
\end{remark}

Finally, we explain \cref{rem:mainthm} \eqref{enum:dim4}, which claimed that \cref{athm:main} is true in dimension 4. If the path components of $\partial M$ are $1$-connected they are homeomorphic to $S^3$ by the now-proven Poincar\'e conjecture and so can be filled with $D^4$'s to obtain a closed 1-connected 4-manifold $M'$. Consider now the fibration sequence \[\Homeo_\partial(M) \lra \Homeo(M') \lra \rm{Emb}^\rm{Top}(\sqcup_k D^4,M')\]
with fibre taken over the canonical inclusion $\sqcup_k D^4 \hookrightarrow M'$ of the discs we used to fill the boundary components. Evaluating an embedding $\sqcup_k D^k$ at the set of origins yields a map
\[\rm{Emb}^\rm{Top}(\sqcup_k D^4,M') \lra \rm{Conf}_k(M').\]
By isotopy extension, this map is a fibration and its fibre over a configuration is the subspace of $\rm{Emb}^\rm{Top}(\sqcup_k D^4,M')$ of those embeddings sending the set of origins to this configuration. Using a 1-parameter family of embeddings $\sqcup_k D^4 \to \sqcup_k D^4$ starting at the identity with increasingly small image, one proves that this deformation retracts onto the subspace of those embeddings sending each $D^k$ into a small neighbourhood around the corresponding point in the configuration, homeomorphic to $D^4$. This can be identified with $\rm{Emb}^\rm{Top}(D^4,D^4)^k$ and is equivalent to $\rm{Top}(4)$ using the fact that $D^4$ is isotopy equivalent to $\bb{R}^4$ and Kister's theorem \cite{Kister}. The upshot is that we have a fibre sequence
\[\rm{Top}(4)^k \lra \rm{Emb}^\rm{Top}(\sqcup_k D^4,M') \lra \rm{Conf}_k(M'),\]
which can be combined with the fact $\pi_i\,\rm{Top}(4)$ is finite for $i \leq 1$ as a consequence of \cite[Section 8.7]{FreedmanQuinn}, to conclude that $\pi_i\,\rm{Emb}^\rm{Top}(\sqcup_k D^4,M')$ is finite for $i \leq 1$. Hence $\pi_0\,\Homeo_\partial(M) \to \pi_0\,\Homeo(M')$ has finite kernel and its image has finite index. The result now follows from \cite[Theorem 1.1]{Quinn}, which says that $\pi_0\,\Homeo(M')$ is arithmetic and hence of finite type, combined with \cref{lem:finite-type-constr} \eqref{enum:ft-finite-index} and \eqref{enum:ft-extensions}. This result also follows from \cite{OrsonPowell}.


\bibliographystyle{amsalpha}
\bibliography{./refs}

\providecommand{\bysame}{\leavevmode\hbox to3em{\hrulefill}\thinspace}
\providecommand{\MR}{\relax\ifhmode\unskip\space\fi MR }
\providecommand{\MRhref}[2]{%
  \href{http://www.ams.org/mathscinet-getitem?mr=#1}{#2}
}
\providecommand{\href}[2]{#2}
\begin{thebibliography}{HMR75}

\bibitem[BG76]{BousfieldGugenheim}
A.~K. Bousfield and V.~K. A.~M. Gugenheim, \emph{On {${\rm PL}$} de {R}ham
  theory and rational homotopy type}, Mem. Amer. Math. Soc. \textbf{8} (1976),
  no.~179, ix+94. \MR{425956}

\bibitem[BHC62]{BorelHarishChandra}
A.~Borel and Harish-Chandra, \emph{Arithmetic subgroups of algebraic groups},
  Ann. of Math. (2) \textbf{75} (1962), 485--535. \MR{147566}

\bibitem[BKK23]{BustamanteKrannichKupers}
M.~Bustamante, M.~Krannich, and A.~Kupers, \emph{Finiteness properties of
  automorphism spaces of manifolds with finite fundamental group}, to appear in
  Math. Ann. (2023), arXiv:2103.13468v2.

\bibitem[BL05]{BlockLazarev}
J.~Block and A.~Lazarev, \emph{Andr\'{e}-{Q}uillen cohomology and rational
  homotopy of function spaces}, Adv. Math. \textbf{193} (2005), no.~1, 18--39.
  \MR{2132759}

\bibitem[Bor91]{Borel}
A.~Borel, \emph{Linear algebraic groups}, second ed., Graduate Texts in
  Mathematics, vol. 126, Springer-Verlag, New York, 1991. \MR{1102012}

\bibitem[Bou75]{Bousfield}
A.~K. Bousfield, \emph{The localization of spaces with respect to homology},
  Topology \textbf{14} (1975), 133--150. \MR{380779}

\bibitem[CE48]{ChevalleyEilenberg}
C.~Chevalley and S.~Eilenberg, \emph{Cohomology theory of {L}ie groups and
  {L}ie algebras}, Trans. Amer. Math. Soc. \textbf{63} (1948), 85--124.
  \MR{24908}

\bibitem[Cer70]{Cerf}
J.~Cerf, \emph{La stratification naturelle des espaces de fonctions
  diff\'{e}rentiables r\'{e}elles et le th\'{e}or\`eme de la pseudo-isotopie},
  Inst. Hautes \'{E}tudes Sci. Publ. Math. (1970), no.~39, 5--173. \MR{292089}

\bibitem[Cha15]{Chatterjee}
P.~Chatterjee, \emph{On abstract homomorphisms of algebraic groups}, 2015,
  arXiv:1511.06180v2.

\bibitem[CP80]{CenklPorter}
B.~Cenkl and R.~Porter, \emph{Malcev's completion of a group and differential
  forms}, J. Differential Geometry \textbf{15} (1980), no.~4, 531--542 (1981).
  \MR{628342}

\bibitem[Dan76]{Dancis}
J.~Dancis, \emph{General position maps for topological manifolds in the
  {${2\over 3}$}rds range}, Trans. Amer. Math. Soc. \textbf{216} (1976),
  249--266. \MR{391098}

\bibitem[DDK81]{DrorDwyerKan}
E.~Dror, W.~G. Dwyer, and D.~M. Kan, \emph{Self-homotopy equivalences of
  virtually nilpotent spaces}, Comment. Math. Helv. \textbf{56} (1981), no.~4,
  599--614. \MR{656214}

\bibitem[ES20]{EspicSaleh}
H.~Espic and B.~Saleh, \emph{On the group of homotopy classes of relative
  homotopy automorphisms}, 2020, arXiv:2002.12083v4, to appear in Journal of
  Homotopy and Related Structures.

\bibitem[Far96]{Farjoun}
E.~D. Farjoun, \emph{Cellular spaces, null spaces and homotopy localization},
  Lecture Notes in Mathematics, vol. 1622, Springer-Verlag, Berlin, 1996.
  \MR{1392221}

\bibitem[FHT01]{FHT}
Y.~F\'{e}lix, S.~Halperin, and J.-C. Thomas, \emph{Rational homotopy theory},
  Graduate Texts in Mathematics, vol. 205, Springer-Verlag, New York, 2001.
  \MR{1802847}

\bibitem[FQ90]{FreedmanQuinn}
M.~H. Freedman and F.~Quinn, \emph{Topology of 4-manifolds}, Princeton
  Mathematical Series, vol.~39, Princeton University Press, Princeton, NJ,
  1990. \MR{1201584}

\bibitem[Hat78]{HatcherConcordance}
A.~E. Hatcher, \emph{Concordance spaces, higher simple-homotopy theory, and
  applications}, Algebraic and geometric topology ({P}roc. {S}ympos. {P}ure
  {M}ath., {S}tanford {U}niv., {S}tanford, {C}alif., 1976), {P}art 1, Proc.
  Sympos. Pure Math., XXXII, Amer. Math. Soc., Providence, R.I., 1978,
  pp.~3--21. \MR{520490}

\bibitem[HMR75]{HiltonMilsinRoitberg}
P.~Hilton, G.~Mislin, and J.~Roitberg, \emph{Localization of nilpotent groups
  and spaces}, North-Holland Mathematics Studies, No. 15, North-Holland
  Publishing Co., Amsterdam-Oxford; American Elsevier Publishing Co., Inc., New
  York, 1975. \MR{0478146}

\bibitem[HW73]{HatcherWagoner}
A.~Hatcher and J.~Wagoner, \emph{Pseudo-isotopies of compact manifolds},
  Ast\'{e}risque, No. 6, Soci\'{e}t\'{e} Math\'{e}matique de France, Paris,
  1973, With English and French prefaces. \MR{0353337}

\bibitem[Kis64]{Kister}
J.~M. Kister, \emph{Microbundles are fibre bundles}, Ann. of Math. (2)
  \textbf{80} (1964), 190--199. \MR{180986}

\bibitem[Kra22]{Krannich}
M.~Krannich, \emph{A homological approach to pseudoisotopy theory. {I}},
  Invent. Math. \textbf{227} (2022), no.~3, 1093--1167. \MR{4384194}

\bibitem[Lee69]{Lees}
J.~A. Lees, \emph{Immersions and surgeries of topological manifolds}, Bull.
  Amer. Math. Soc. \textbf{75} (1969), 529--534. \MR{239602}

\bibitem[Lur09]{LurieHTT}
J.~Lurie, \emph{Higher topos theory}, Annals of Mathematics Studies, vol. 170,
  Princeton University Press, Princeton, NJ, 2009. \MR{2522659}

\bibitem[Maj00]{Majewski}
M.~Majewski, \emph{Rational homotopical models and uniqueness}, Mem. Amer.
  Math. Soc. \textbf{143} (2000), no.~682, xviii+149. \MR{1751423}

\bibitem[Mar96]{Maruyama}
K.~Maruyama, \emph{Finitely presented subgroups of the self-homotopy
  equivalences group}, Math. Z. \textbf{221} (1996), no.~4, 537--548.
  \MR{1385167}

\bibitem[Nic82]{Nicas}
A.~J. Nicas, \emph{Induction theorems for groups of homotopy manifold
  structures}, Mem. Amer. Math. Soc. \textbf{39} (1982), no.~267, vi+108.
  \MR{668807}

\bibitem[OP22]{OrsonPowell}
P.~Orson and M.~Powell, \emph{Mapping class groups of simply connected
  4-manifolds with boundary}, 2022, arXiv:2207.05986v1.

\bibitem[Qui69]{QuillenRational}
D.~Quillen, \emph{Rational homotopy theory}, Ann. of Math. (2) \textbf{90}
  (1969), 205--295. \MR{258031}

\bibitem[Qui86]{Quinn}
F.~Quinn, \emph{Isotopy of {$4$}-manifolds}, J. Differential Geom. \textbf{24}
  (1986), no.~3, 343--372. \MR{868975}

\bibitem[Sch80]{Scheerer}
H.~Scheerer, \emph{Arithmeticity of groups of fibre homotopy equivalence
  classes}, Manuscripta Math. \textbf{31} (1980), no.~4, 413--424. \MR{586133}

\bibitem[Seg83]{Segal}
D.~Segal, \emph{Polycyclic groups}, Cambridge Tracts in Mathematics, vol.~82,
  Cambridge University Press, Cambridge, 1983. \MR{713786}

\bibitem[Ser79]{Serre}
J.-P. Serre, \emph{Arithmetic groups}, Homological group theory ({P}roc.
  {S}ympos., {D}urham, 1977), London Math. Soc. Lecture Note Ser., vol.~36,
  Cambridge Univ. Press, Cambridge-New York, 1979, pp.~105--136. \MR{564421}

\bibitem[Ser94]{SerreGalois}
\bysame, \emph{Cohomologie galoisienne}, fifth ed., Lecture Notes in
  Mathematics, vol.~5, Springer-Verlag, Berlin, 1994. \MR{1324577}

\bibitem[Sie72]{Siebenmann}
L.~C. Siebenmann, \emph{Deformation of homeomorphisms on stratified sets. {I},
  {II}}, Comment. Math. Helv. \textbf{47} (1972), 123--136; ibid. 47 (1972),
  137--163. \MR{319207}

\bibitem[Sul77]{Sullivan}
D.~Sullivan, \emph{Infinitesimal computations in topology}, Inst. Hautes
  \'{E}tudes Sci. Publ. Math. (1977), no.~47, 269--331 (1978). \MR{646078}

\bibitem[Tri95]{TriantafillouFinite}
G.~Triantafillou, \emph{Automorphisms of spaces with finite fundamental group},
  Trans. Amer. Math. Soc. \textbf{347} (1995), no.~9, 3391--3403. \MR{1316864}

\bibitem[Tri99]{TriantafillouTelAviv}
\bysame, \emph{The arithmeticity of groups of automorphisms of spaces}, Tel
  {A}viv {T}opology {C}onference: {R}othenberg {F}estschrift (1998), Contemp.
  Math., vol. 231, Amer. Math. Soc., Providence, RI, 1999, pp.~283--306.
  \MR{1707350}

\bibitem[Wal99]{Wall}
C.~T.~C. Wall, \emph{Surgery on compact manifolds}, second ed., Mathematical
  Surveys and Monographs, vol.~69, American Mathematical Society, Providence,
  RI, 1999, Edited and with a foreword by A. A. Ranicki. \MR{1687388}

\bibitem[Wil76]{Wilkerson}
C.~W. Wilkerson, \emph{Applications of minimal simplicial groups}, Topology
  \textbf{15} (1976), no.~2, 111--130. \MR{402737}

\bibitem[Wil80]{WilkersonErrata}
\bysame, \emph{Errata: ``{A}pplications of minimal simplicial groups''
  [{T}opology {\bf 15} (1976), no. 2, 111--130; {MR} {\bf 53} \#6551]},
  Topology \textbf{19} (1980), no.~1, 99. \MR{559480}

\end{thebibliography}

\bigskip

\end{document}